\documentclass[11pt,paper=a4]{article}
 \usepackage{indentfirst, latexsym,bm}
 \usepackage{amsmath}
 \usepackage{pifont}
 \usepackage{amsfonts}
 \usepackage{mathrsfs}
 \usepackage{array}
 \usepackage{multirow} 
 \usepackage{graphicx}
 \usepackage{subfigure}
 \usepackage{picinpar}
 \usepackage{setspace}
 \usepackage[textsize=footnotesize]{todonotes}
 \RequirePackage[colorlinks,citecolor=blue,urlcolor=blue,linkcolor=blue]{hyperref}
 \usepackage[top=2cm,bottom=2cm, outer=2cm, inner=2cm]{geometry}
 
 \usepackage{caption}
 \usepackage{subcaption}
 \usepackage[labelfont={bf,small},textfont={small}]{caption}
 
 \RequirePackage[numbers]{natbib}
 \usepackage{enumitem} 
 
 \usepackage{booktabs}
 \usepackage{authblk}
 \usepackage{float}
 \usepackage{threeparttable}
 \usepackage{mathrsfs,amsfonts,amsmath}
 
 \usepackage{color}
 
 \makeatletter
 \def\namedlabel#1#2{\begingroup
 	#2%
 	\def\@currentlabel{#2}%
 	\phantomsection\label{#1}\endgroup
 }
 \makeatother

 \numberwithin{figure}{section}

 \newcommand\email[1]{\href{mailto:#1}{ \nolinkurl{#1}}}

 \newtheorem{theorem}{Theorem}[section]
 \newtheorem{definition}[theorem]{Definition}
 \newtheorem{lemma}[theorem]{Lemma}
 \newtheorem{corollary}[theorem]{Corollary}
 \newtheorem{proposition}[theorem]{Proposition}
 \newtheorem{remark}[theorem]{Remark}
 \newtheorem{condition}[theorem]{Condition}
 \newtheorem{example}{Example}[section]

 \def\blemma{\begin{lemma}}\def\elemma{\end{lemma}}
 \def\bproposition{\begin{proposition}}\def\eproposition{\end{proposition}}
 \def\ttheorem{\begin{theorem}}\def\etheorem{\end{theorem}}
 \def\bcorollary{\begin{corollary}}\def\ecorollary{\end{corollary}}
 \def\bremark{\begin{remark}}\def\eremark{\end{remark}}
 \def\bcondition{\begin{condition}}\def\econdition{\end{condition}}

 \def\benumerate{\begin{enumerate}}\def\eenumerate{\end{enumerate}}
 \def\bitemize{\begin{itemize}}\def\eitemize{\end{itemize}}

 \def\beqlb{\begin{eqnarray}}\def\eeqlb{\end{eqnarray}}
 \def\beqnn{\begin{eqnarray*}}\def\eeqnn{\end{eqnarray*}}
 \def\ar{\!\!\!&}

 \def\proof{\noindent{\it Proof.~~}}\def\qed{\hfill$\Box$\medskip}

\begin{document} 
  \title{\bf  \Large Asymptotic Results for Spectrally Positive Compound Poisson Processes.}
 
 \author{Zhi-Hao Cui\footnote{School of Mathematical Sciences, Nankai University, China; email: cuizh.math@gmail.com} \quad  \  Hao Wu\footnote{School of Mathematical Sciences, Nankai University, China; email: wuhao.math@outlook.com.} 
 }  
 \maketitle
 \begin{abstract} 
Finite excursions away from zero of a spectrally positive compound Poisson process with a negative drift  can always be decomposed into two parts lying above and below zero, respectively. 
This paper is concerned with the asymptotic relationships among the lengths and heights of these two parts. 
Our results state that both their lengths and heights are asymptotically strongly dependent and exhibit a scale symmetry.
 	\bigskip
 	
 	
 	\smallskip
 	

 \end{abstract}

  \section{Introduction and main results} 
\label{Sec.Introduction}
\setcounter{equation}{0}
Consider a spectrally positive compound Poisson process $X=\{ X_t:t\geq 0 \}$ with negative drift and Laplace exponent given by 
\beqlb\label{eqn2.01}
\psi(\lambda):=\log\mathbf{E}[\exp\{-\lambda X_1\}]=b\lambda+\int_0^\infty(e^{-\lambda x}-1 )\,\nu(dx), \quad \lambda \geq 0,
\eeqlb
for some constant $b>0$ and finite measure $\nu(dx)$ on $(0,\infty)$. 
The function $ \psi$ is infinitely differentiable and strictly convex on $(0, \infty)$. Its  right inverse is defined by
$
\Phi(q)=\sup \{\lambda \geq 0: \psi(\lambda)=q\} $ for $ q\geq 0.
$  For $x\in\mathbb{R}$, we write $\tau_{x}^-$ and $\tau_{x}^+$ for the first passage times of $X$ 
into $(-\infty,x)$ and $(x,\infty)$, respectively, and define 
\beqnn
T_0:=\inf\{t>0:X_t=0\}
\eeqnn
the first hitting time of zero. Here we make the convention that $\inf \emptyset =\infty$. 
Additionally, we also consider the local maximum and local minimum during the time $[0,T_0]$ defined by 
\beqnn
H_0:=\sup_{0\le s\le T_0} X_s
\quad\text{and}\quad
-\widehat{H}_0:=\inf_{0\le s\le T_0} X_s.
\eeqnn

Notice that the measure of finite excursions  of $X$ away from $0$ is proportional to the law of $(X_t, 0\leq t\leq T_0)$ conditioned on $T_0<\infty$; see Figure~\ref{Figure01}. 
In this paper,  we tempt to understand the asymptotic dependence of  the two parts of excursions lying above and below $0$ by studying the asymptotics of the joint distribution of $(\tau_0^+,T_0 - \tau_0^+,H_0,\widehat{H}_0)$. 

\begin{center}
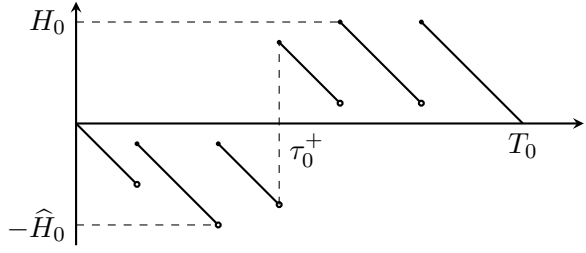

	\begin{tikzpicture}[scale=0.27, x=1cm, y=1cm, >=stealth]
		
		\draw[->, line width=0.8pt] (0,-6) -- (0,6);
		\draw[->, line width=0.8pt] (0,0) -- (25,0);
		
		\def\taup{10}
		\def\Tzero{22}
		
		\draw[dashed] (0,5) -- (13,5);
		\draw[dashed] (0,-5) -- (7,-5);
		
		\node[left] at (0,5) {$H_0$};
		\node[left] at (0,-5) {$-\widehat{H}_0$};
		
		\draw[line width=0.8pt, domain=0:3]
		plot (\x, {-\x});
		
		\draw[line width=0.8pt, fill=white] (3,-3) circle (0.12);
		\fill (3,-1) circle (0.12);
		
		\draw[line width=0.8pt, domain=3:7]
		plot (\x, {2-\x});
		
		\draw[line width=0.8pt, fill=white] (7,-5) circle (0.12);
		\fill (7,-1) circle (0.12);
		
		\draw[line width=0.8pt, domain=7:10]
		plot (\x, {6-\x});
		
		\draw[line width=0.8pt, fill=white] (10,-4) circle (0.12);
		\draw[dashed] (10,-4) -- (10,4);
		\fill (10,4) circle (0.12);
		
		\node[below right] at (\taup,0) {$\tau_0^+$};
		
		\draw[line width=0.8pt, domain=10:13]
		plot (\x, {14-\x});
		
		\draw[line width=0.8pt, fill=white] (13,1) circle (0.12);
		\fill (13,5) circle (0.12);
		
		\draw[line width=0.8pt, domain=13:17]
		plot (\x, {18-\x});
		
		\draw[line width=0.8pt, fill=white] (17,1) circle (0.12);
		\fill (17,5) circle (0.12);
		
		\draw[line width=0.8pt, domain=17:\Tzero]
		plot (\x, {\Tzero-\x});
		
		\node[below] at (\Tzero,0) {$T_0$};
		
	\end{tikzpicture}
	
	\captionof{figure}{\small A typical excursion away from $0$ that can be decomposed into two parts lying above and below zero.}\label{Figure01}
	
\end{center}

\textit{\textbf {The recurrent  case.}}  In this case, we have $\mathbf{E}[X_1]=0$ and $\tau_0^{+}<\infty$ a.s. To obtain the asymptotic results for the  two parts of excursion,  we impose the following standard condition on $X$: 
\begin{condition}\label{Assumption01}
	The process $X$ belongs to the domain of attraction of a stable law without centering, that is, there exists a positive function $\boldsymbol{c}$ on $\mathbb{R}_+$ such that 
	\beqnn
	X_t/\boldsymbol{c}(t) \ar\overset{\rm d}\to\ar Y_1, \quad \mbox{as $t\to\infty$},
	\eeqnn
	where $Y$ is a strictly stable process with index $\alpha\in (0,2)$ and negativity parameter $\rho= \mathbf{P} \big(Y_1\leq 0\big)\in(0,1)$. 
\end{condition}

For $\kappa\in\mathbb{R}$, we write $r \in \mathrm{RV}^\infty_\kappa$ (resp. $r\in \mathrm{RV}^0_\kappa$) if $r(tx)/r(t)\to x^\kappa$ as $ t\to\infty $  (resp. as $t\to 0+$) for any $x>0$. 
By   \cite[p.~218]{Bertoin1996}, Proposition~6 and Theorem~VI.14  in \cite[p.~192 and p.~169]{Bertoin1996}, Condition~\ref{Assumption01} implies  that 
\beqnn
\alpha\rho=1,\quad \alpha\in(1,2)
\quad \mbox{and}\quad
\Phi(q)\in \mathrm{RV}_{\rho}^{0}.
\eeqnn 
Moreover,  we  write 	 $\overline{\nu}(h):=\nu[h,\infty)$ for the \textit{tail-function} of $\nu(dx)$, which belongs to $\mathrm{RV}_{-\alpha}^{\infty}$.
Our first result studies  the asymptotic behavior of the tail probabilities for the  lengths $(\tau_0^+,T_0 - \tau_0^+)$ and heights $(H_0,\widehat{H}_0)$.
\begin{theorem}\label{Mainresult02}
	The following  hold as $t,h\to\infty$: 
	\begin{enumerate}
		\item[(1)]   
		$\mathbf{P}(  \tau_0^{+}>t)=\mathbf{P}(  T_0-\tau_0^{+}>t)
		\sim  \dfrac{1}{b\Gamma(\rho)}\cdot \dfrac{1}{t\Phi(1/t)} ;$ \label{main1}
		\item[(2)] 	$\mathbf{P}(  H_0>h) =\mathbf{P}(  \widehat{H}_0>h) 
		\sim b^{-1}\,\Gamma(2-\alpha)\, \Gamma(\alpha-1)\cdot  h\overline{\nu}(h)$.
	\end{enumerate}
\end{theorem}

The next theorem provides an exact description on the asymptotic dependence between  ${T_0-\tau_0^+}$ and ${\tau_0^+}$, as well as between ${H_0}$ and ${\widehat{H}_0}$. 
To formulate it, we need to define two non-negative functions $f$ and $g$ on $(0,\infty)$ by
\beqlb\label{eqn1.2}
f(x):=\frac{x^{\rho-1}}{\Gamma(\rho)}
+\frac{1}{\Gamma(\rho)} \frac{\sin(\pi \rho)}{\pi}
\int_0^x 
\frac{(x-u)^{\rho-1} (1+u) u^{\rho-1}}{u^{2\rho}-2u^\rho \cos(\pi \rho)+1}du
\eeqlb
and 
\beqlb\label{eqn1.3}
g(x):= \frac{\alpha-1}{\Gamma(2-\alpha)\Gamma(\alpha-1)}\int_{0}^{1}s^{\alpha-2}ds\int_{0}^{1}t^{\alpha-2}(x+1-xt-s)^{1-\alpha}dt.
\eeqlb

\begin{theorem}\label{MainThm01}
	For any $a>0$,	the following  hold as $t,h\to\infty$:
	\begin{enumerate}
		\item[(1)]   
		$
		\mathbf{P}\big(T_0-\tau_{0}^{+}>at \,|\, \tau_{0}^{+}>t\big)=	\mathbf{P}\big(\tau_{0}^{+}>at \,|\, T_0-\tau_{0}^{+}>t\big)\to a^{\rho-1}+1-\Gamma(\rho) f(a)\in (0,1);
		$
		\item[(2)] $
		\mathbf{P}\big(H_0>ah \,|\, \widehat{H}_0>h\big)	=\mathbf{P}\big(\widehat{H}_0>ah \,|\, H_0>h\big)	\to g(a)\in(0,1)$.
	\end{enumerate}
\end{theorem}

In particular, Theorem \ref{MainThm01} with $a=1$ tells that $2-\Gamma(\rho)f(1)\in(0,1)$, which implies that the two events 
$T_0-\tau_0^+>t$ and $\tau_0^+>t$ are asymptotically dependent and enjoy the asymptotic symmetry in the sense that 
\beqnn
\mathbf{P} \Big( \frac{T_0-\tau_0^+}{t}\,\Big|\,\tau_0^+>t  \Big) 
\sim \mathbf{P} \Big( \frac{\tau_0^+}{t}\,\Big|\,T_0-\tau_0^+>t \Big) .
\eeqnn
Analogous asymptotic dependence also can be observed between $H_0>h$ and $\widehat{H}_0>h$.

In the extreme value theory, the well-known one-big-jump principle states that for two  weakly dependent  heavy-tailed random variables $\xi_1$ and $\xi_2$, we have as $x\to\infty$,
\beqnn
\mathbf{P}( \xi_1\vee \xi_2> x ) \sim \mathbf{P}( \xi
_1 > x ) + \mathbf{P}(\xi_2> x ).
\eeqnn 
Comparing to this principle, a direct consequence of the preceding theorems further identifies that  the two events  $T_0-\tau_{0}^{+}>t$ and  $\tau_0^+>t$ are strongly dependent. The same is also true for the events $H_0>h$ and $\widehat{H}_0>h$. 
\begin{corollary}
	The following  hold as $t,h\to\infty$: 
	\begin{enumerate}
		\item[(1)]   
		$
		\mathbf{P}\big(T_0-\tau_{0}^{+}>t \,|\, \tau_{0}^{+}\vee (T_0-\tau_{0}^{+})>t\big)=	\mathbf{P}\big(\tau_{0}^{+}>t \,|\, \tau_{0}^{+}\vee (T_0-\tau_{0}^{+})>t\big)\to \dfrac{1}{\Gamma(\rho) f(1)}\in (\frac{1}{2},1);
		$
		\item[(2)] $
		\mathbf{P}\big(H_0>h \,|\, \widehat{H}_0\vee H_{0}>h\big)	=\mathbf{P}\big(\widehat{H}_0>h \,|\, \widehat{H}_0\vee H_{0}>h\big)	\to\dfrac{1}{2-g(1)} \in(\frac{1}{2},1)$.
	\end{enumerate}
\end{corollary}

\textit{\textbf {The transient case.}} 
We now turn to the case where $X$ has negative expectation $\mathbf{E}[X_1]=-\beta\in(-\infty,0 )$
and hence drifts to $-\infty$. In this case, Theorem~17 in \cite[p.~204]{Bertoin1996} shows that
\beqlb\label{aa2.9}
\mathbf{P}(\tau_0^+<\infty)=1-\beta/b<1.
\eeqlb
Our results are established under the regular-variation condition on the tail-function $\overline{\nu}$  defined before Theorem~\ref{Mainresult02}.
\begin{condition}\label{Assumption02}  Assume that $\overline{\nu} \in\mathrm{RV}_{-\theta}^{\infty}$ for some $\theta>1$.
\end{condition}

The first  result studies the asymptotic behavior of the joint-probabilities of the  lengths $(\tau_0^+, T_0 - \tau_0^+)$ and heights $(H_0,\widehat{H}_0)$.
\begin{theorem}\label{Mainresult005}
	The following   hold as $t,h\to\infty$:
	\begin{enumerate}
		\item[(1)] $	
		\mathbf{P}(  \tau_0^{+}>t\,|\, \tau_0^{+}<\infty)=	\mathbf{P}(  T_0-\tau_0^{+}>t\,|\, \tau_0^{+}<\infty)
		\sim \dfrac{1}{\theta -1}\dfrac{1}{b-\beta}\cdot \beta t\overline{\nu}(\beta t);$ 
		\item[(2)] $\mathbf{P}(  H_0>h\,|\, \tau_0^{+}<\infty) =\mathbf{P}(  \widehat{H}_0>h\,|\, \tau_0^{+}<\infty) 
		\sim \dfrac{1}{\theta -1}\dfrac{1}{b-\beta}\cdot h\overline{\nu}(h) .$
	\end{enumerate}
\end{theorem}

Analogous to  the    recurrent case, the second result states that  the events
$\tau_0^{+}>t$ and $T_0- \tau_0^{+}>t$ 
are  asymptotically dependent.  The same result also holds for the events
$H_0>h$ and $\widehat{H}_0>h$. Moreover, the asymptotic symmetry described above also holds.

\begin{theorem}\label{MainThm02}
	For any $a>0$,	the following hold as $t,h\to\infty$: 
	\begin{enumerate}
		\item[(1)]  
		$
		\mathbf{P}\big(T_0-\tau_{0}^{+}>at \,|\, \tau_{0}^{+}>t, \tau_0^{+}<\infty\big)=	\mathbf{P}\big(\tau_{0}^{+}>at \,|\, T_0-\tau_{0}^{+}>t, \tau_0^{+}<\infty\big)\to (1+a)^{1-\theta}\in(0,1);$
		\item[(2)] $	\mathbf{P}\big(H_0>ah \,|\, \widehat{H}_0>h,\tau_{0}^{+}<\infty\big)=\mathbf{P}\big(\widehat{H}_0>ah \,|\, H_0>h,\tau_{0}^{+}<\infty\big)\to(1+a)^{1-\theta}\in(0,1)
		$.
	\end{enumerate}
\end{theorem}

The following corollary is an immediate consequence of the preceding two theorems.
\begin{corollary}
	The following  hold as $t,h\to\infty$:
	\begin{enumerate}
		\item[(1)]   
		$
		\mathbf{P}\big(T_0-\tau_{0}^{+}>t \,|\, \tau_{0}^{+}\vee (T_0-\tau_{0}^{+})>t\big)=	\mathbf{P}\big(\tau_{0}^{+}>t \,|\, \tau_{0}^{+}\vee (T_0-\tau_{0}^{+})>t\big)\to \dfrac{1}{2-(1+a)^{1-\theta}}\in (\frac{1}{2},1);
		$
		\item[(2)] $
		\mathbf{P}\big(H_0>h \,|\, \widehat{H}_0\vee H_{0}>h\big)	=\mathbf{P}\big(\widehat{H}_0>h \,|\, \widehat{H}_0\vee H_{0}>h\big)	\to\dfrac{1}{2-(1+a)^{1-\theta}} \in(\frac{1}{2},1)$.
	\end{enumerate}
\end{corollary}

\textbf{Organization of this paper.}
In Section \ref{Sec.Preliminaries}, we  provide two auxiliary results that play an important role in the proofs of our main results. 
Section \ref{Sec.TransientCase} is devoted to the proofs of Theorem \ref{Mainresult02}, \ref{MainThm01}, \ref{Mainresult005} and  \ref{MainThm02}.

 \section{Auxiliary results} 
\label{Sec.Preliminaries}
\setcounter{equation}{0}
In  this section, we  present several auxiliary results that will be essential for proving our main results.
All processes in this work are defined on a complete probability space
$(\Omega,\mathscr{F},\mathbf{P})$ equipped with a filtration $\{\mathscr{F}_t\}_{t\geq 0}$ satisfying the usual hypotheses. 
For $x\in\mathbb{R}$, let $\mathbf{P}_x$ and $\mathbf{E}_x$ denote the law and expectation of the  process starting from $x$. 
For simplicity, we also write $\mathbf{P}= \mathbf{P}_0$ and $\mathbf{E}=\mathbf{E}_0$.

\begin{lemma}\label{lem3.2}
	For any  $q_1,q_2, h_1,h_2\geq 0$,	the law of the pair $\big(\tau_{0}^+, T_0-\tau_{0}^+,H_0,\hat{H}_0\big)$ is given by
	\beqlb\label{eqn2.1}
	\lefteqn{\mathbf{E}\big[
		\exp\{-q_1\tau_{0}^+ -q_2(T_0-\tau_{0}^+)\}; H_0>h_1,\hat{H}_0>h_2 
		, \tau_0^{+}<\infty
		\big]}\ar\ar\cr
	\ar=\ar b^{-1}  \int_{0}^{\infty} 
	\mathbf{E}_{x}\big[e^{-q_1\tau_0^{-}}; H_0>h_1\big]dx\int_{0}^{\infty} \mathbf{E}_{y}\big[e^{-q_2\tau_0^{-}}; H_0>h_2\big]
	\nu(dy+x).
	\eeqlb 
\end{lemma}
\proof
Note that the left-hand side of \eqref{eqn2.1} equals to 
\beqlb\label{eqn2.2}
\ar\ar\iint_{\mathbb{R}_+^2}\mathbf{E}\big[
\exp\{-q_1\tau_{0}^+ -q_2(T_0-\tau_{0}^+)\}; H_0>h_1,\hat{H}_0>h_2
\,\big|\, X_{\tau_{0}^{+}-}=-x , X_{\tau_{0}^{+}}= y, \tau_0^{+}<\infty
\big]\cr\cr
\ar\ar\hspace{23em}\times\mathbf{P}(X_{\tau_{0}^{+}-}\in -dx , X_{\tau_{0}^{+}}\in dy, \tau_{0}^{+}<\infty).
\eeqlb
It follows from  Theorem 17 in \cite[p204]{Bertoin1996} that 
\beqlb\label{eqn2.22}
\mathbf{P}(X_{\tau_{0}^{+}-}\in -dx , X_{\tau_{0}^{+}}\in dy,\tau_{0}^{+}<\infty)= b^{-1} dx\,  \nu(dy+x),\quad x,y\geq 0.
\eeqlb
By Lemma 1 in \cite{bertoin1992extension}, conditionally on
$X_{\tau_{0}^{+}-}=-x$ and $X_{\tau_{0}^{+}}=y$, the processes
$(-X_{(\tau_{0}^{+}-t)-},\,0\le t<\tau_0^+)$
and
$(X_{\tau_0^++t},\,0\le t<T_0-\tau_0^+)$ are independent; moreover, they have the same laws as $X$ started
from $x$ and $y$, respectively, and  killed upon hitting $0$.
Hence, combining this  together with \eqref{eqn2.22}  and then taking them back into \eqref{eqn2.2}, we get the desired result.
\qed
\begin{corollary}\label{cor3.3}
	For any  $q_1,q_2\geq 0$, we have
	\beqnn
	\mathbf{E}\big[
	\exp\{-q_1\tau_{0}^+ -q_2(T_0-\tau_{0}^+)\}
	;\tau_0^{+}<\infty
	\big]=  1-b^{-1}\frac{q_1-q_2}{\Phi(q_1)-\Phi(q_2)}.
	\eeqnn
\end{corollary}
\proof 
It follows from  Theorem 3.12 in \cite[p.85]{Kyprianou2014} that for any $q>0$ and $x\geq 0$, $\mathbf{E}_x\big[e^{-q\tau_0^-}\big]=e^{-\Phi(q)x}$. 
Combining this  with Lemma~\ref{lem3.2} yields  
\beqnn
\mathbf{E}\big[
\exp\{-q_1\tau_{0}^+ -q_2(T_0-\tau_{0}^+)\}
;\tau_0^{+}<\infty
\big]	\ar=\ar b^{-1}   \int_{0}^{\infty}
e^{-\Phi(q_1)x}dx\int_{0}^{\infty} e^{-\Phi(q_2)y}
\nu(dy+x).
\eeqnn
Making the change of variables
$z=x+y,x=x$ and then  changing the order of integration, we see that the double integral reduces to  
$ \int_{0}^{\infty}\big(
e^{-\Phi(q_2)x}- e^{-\Phi(q_1)x}\big)
\nu(dx)/(\Phi(q_1)-\Phi(q_2)),
$
which equals $b - (q_1 - q_2)/( \Phi(q_1) - \Phi(q_2) )$ by  \eqref{eqn2.01}. 
Combining the previous results completes the proof.
\qed

  \section{Proof of main results} 
\label{Sec.TransientCase}
\setcounter{equation}{0}
The equalities in Theorems~\ref{Mainresult02}, \ref{MainThm01}, \ref{Mainresult005} and \ref{MainThm02} follow directly from Lemma~\ref{lem3.2}. Hence, in the following proofs, we focus on the asymptotic properties.
\subsection{ The recurrent case.}
In this section, we prove  the main asymptotic results for  compound Poisson processes under Condition \ref{Assumption01}.
We first recall several auxiliary properties of $X$.
The scale function $W=\{W(x): x\in\mathbb{R} \}$ associated to $\psi$ is a non-negative function, which is identically zero on $(-\infty,0)$ and characterized on $[0,\infty)$ as a continuous, strictly increasing function whose Laplace transform is given by $1/\psi(q)$, $q>0$.
This along with integration by parts  shows that  
\beqnn
\int_0^\infty  e^{-q x}\, dW(x) 
=  \int_0^\infty \lambda e^{-q x} \big(W(x)-W(0)\big)\, dx 
=\frac{q}{ \psi(q )}-W(0),
\eeqnn
which goes to $\infty$ as $q \to 0$ and also belongs to $\mathrm{RV}^0_{1-\alpha}$; see Proposition 6 in  \cite[p.192]{Bertoin1996}.
An alternative representation of $W$ is provided by identity~(8.22) in
\cite[p.~239]{Kyprianou2014}, namely,
\beqlb\label{eqn.500}
W(x)=  W(1)\exp\Big\{ \int_1^x \underline{n}(\overline\epsilon>y)\, dy \Big\} ,\quad x\geq 0,
\eeqlb
where $\underline n$ denotes the excursion measure of $X$ reflected at its past infimum and $\overline\epsilon$ denotes the excursion height.
By the Tauberian theorem; see \cite[p.10]{Bertoin1996}, we have $W \in \mathrm{RV}^\infty_{\alpha-1}$.  
Hence from  \eqref{eqn.500} and the Karamata representation theorem; see Theorem~1.3.1 in \cite[p.12]{BinghamGoldieTeugels1987},  we have  as $h\to\infty$,
\beqlb\label{a2}
\underline{n}\big(\overline{\epsilon}>h\big) \sim (\alpha-1)\cdot h^{-1}.
\eeqlb
The two-sided exit identity (see \cite[p.~234]{Kyprianou2014}) gives
\beqlb\label{eqn2.5}
\mathbf{P}_{x}(H_0>h)=1-\mathbf{P}_x( \tau_h^{+}>\tau_0^{-})
=1-W(h-x)/W(h), \quad h\geq  x>0.
\eeqlb

\subsubsection{Proof of Theorem \ref{Mainresult02}} 
For the first claim, by Corollary \ref{cor3.3}, 
$
\mathbf{E}\big[1-e^{-q\tau_0^+}\big]=b^{-1}\cdot  q/ \Phi	(q)\in\mathrm{RV}_{1-\rho}^{0}.
$
Then applying Karamata's Tauberian theorem (see Corollary 8.1.7 in \cite[p.334]{BinghamGoldieTeugels1987})  yields the first claim.
We now turn to prove the second claim. From Lemma~\ref{lem3.2}  and \eqref{eqn2.5}, we have for any $\delta>0$,  
\beqnn
\mathbf{P}\big(H_0>h\big)
= b^{-1} \int_{0}^{\infty}\mathbf{P}_{x}(H_0>h)\overline{\nu}(x)dx=b^{-1}\sum_{i=1}^{3}I_{i}^{\delta},
\eeqnn
where
$
I_{1}^{\delta}=\int_{0}^{\delta h}\big(1-W(h-x)/W(h)\big)\overline{\nu}(x)dx 
$,  $I_{2}^{\delta}=\int_{\delta h}^{h}\big(1-W(h-x)/W(h)\big)\overline{\nu}(x)dx
$
and 
\beqnn I_{3}^{\delta}=\int_{h}^{\infty}\big(1-W(h-x)/W(h)\big)\overline{\nu}(x)dx.\eeqnn
Firstly,    \eqref{eqn.500} induces that
\beqnn
I_{1}^{\delta}\leq \int_{0}^{\delta h}\overline{\nu}(x)dx\int_{h-x}^{h}\underline{n}(\overline{\epsilon}>y)dy\leq \underline{n}(\overline{\epsilon}>(1-\delta )h) \int_{0}^{\delta h}x\overline{\nu}(x)dx.
\eeqnn
From \eqref{a2} and  Proposition 1.5.8 in \cite[p.26]{BinghamGoldieTeugels1987}, we  get  as $h\to\infty$,
$
\underline{n}(\overline{\epsilon}>(1-\delta )h)\sim   (\alpha - 1)/((1-\delta)h)$  and  
$\int_{0}^{\delta h}x\overline{\nu}(x)dx\sim (\delta h)^2 \overline{\nu}(\delta h) /(2-\alpha).
$
Based on these two asymptotic equivalences, we obtain that
\beqlb\label{eqn4.28}
\lim_{\delta\to 0+}\lim_{h\to\infty}\frac{I_{1}^{\delta}}{h\overline{\nu}(h)}=\lim_{\delta\to0+} \frac{\alpha-1}{2-\alpha}\cdot\frac{\delta^{2-\alpha}}{1-\delta}=0.
\eeqlb
Secondly, it is easy to see that 
$
I_{2}^{\delta}=  \int_{\delta}^{1}\big(1-W(h(1-y))/W(h)\big)\cdot\overline{\nu}(hy)/\overline{\nu}(h)dy.
$
Recall that $W \in \mathrm{RV}^\infty_{\alpha-1}$ and $\overline{\nu} \in \mathrm{RV}^\infty_{-\alpha}$, then  by  Theorem 1.5.2 in \cite[p.22]{BinghamGoldieTeugels1987}, we induce that uniformly in  $y\in(\delta,1)$, as $h \to\infty$, 
\beqlb\label{a5}
1-W(h(1-y))/W(h)\to  1-(1-y)^{\alpha-1}\quad\text{and}\quad \overline{\nu}(hy)/\overline{\nu}(h)\to y^{-\alpha}.
\eeqlb
Hence, we obtain
\beqnn
\lim_{\delta\to 0+}\lim_{h\to\infty}\frac{I_{2}^{\delta}}{h\overline{\nu}(h)}= \int_{0}^{1}\frac{1-(1-y)^{\alpha-1}}{y^{\alpha}}dy= -\frac{1}{\alpha-1}+\Gamma\big(2-\alpha\big)\Gamma\big(\alpha-1\big).
\eeqnn
Finally,  
by  Proposition~1.5.9 in \cite[p.~27]{BinghamGoldieTeugels1987}, we have  as $h\to\infty$, 
\beqlb\label{eqn2.15}
I_{3}^{\delta}=\int_{h}^{\infty}\overline{\nu}(x)dx\sim\frac{h\overline{\nu}(h)}{\alpha-1}.
\eeqlb
Putting these three estimates together, we have finished the proof.
\qed
\subsubsection{Proof of Theorem \ref{MainThm01}}

\noindent\textbf{Proof of Theorem \ref{MainThm01}(1).}
Note that 
\beqnn
\mathbf{P}(T_0-\tau_{0}^{+}>at\,|\,\tau_{0}^{+}>t)=1+\frac{\mathbf{P}(T_0-\tau_{0}^{+}>at)}{\mathbf{P}(\tau_{0}^{+}>t)}-\frac{1-\mathbf{P}(T_0-\tau_{0}^{+}\leq at, \tau_{0}^{+}\leq t)}{\mathbf{P}(\tau_{0}^{+}>t)}.
\eeqnn
From Theorem   \ref{Mainresult02}(1), we obtain  as $t\to\infty$,
\beqnn
\frac{\mathbf{P}(T_0-\tau_{0}^{+}>at)}{\mathbf{P}(\tau_{0}^{+}>t)}\to a^{\rho-1}.
\eeqnn
It remains to consider the asymptotic behavior of $1-\mathbf{P}(T_0-\tau_{0}^{+}\leq at, \tau_{0}^{+}\leq t)$.
Denote by 
\beqnn
K(x,y):=\int_{0}^{x}du\int_{0}^{y}\big(1-\mathbf{P}(\tau_{0}^{+}\leq u,T_0-\tau_{0}^{+}\leq v)\big)dv.
\eeqnn
From Corollary  \ref{cor3.3}, we have
\beqnn
\widehat{K}(q_1,q_2):=q_1q_2\iint_{\mathbb{R}_+^2}e^{-q_1 x-q_2 y}K(x,y)dxdy\ar=\ar\frac{1-\mathbf{E}\big[\exp\{-q_1 \tau_{0}^{+}-q_2(T_0-\tau_{0}^{+})\}\big]}{q_1q_2}=  \frac{b^{-1}}{q_1q_2}\cdot \frac{q_1-q_2}{\Phi(q_1)-\Phi(q_2)}.
\eeqnn
Therefore,  as $t\to\infty$,
\beqnn
\widehat{K}\big(q_1/t,q_1/t\big)=b^{-1} \frac{t^2}{q_1q_2}\cdot \frac{q_1/t-q_2/t}{\Phi(q_1/t)-\Phi(q_2/t)}\sim b^{-1}\frac{q_1-q_2}{q_1q_2(q_1^{\rho}-q_2^{\rho})}\cdot\frac{t}{\Phi(1/t)}.
\eeqnn
By Theorem 2.4 in \cite{de1984domains},  as $t\to\infty$,
\beqlb\label{a3}
K(tx,ty)\sim b^{-1} \lambda(x,y)\cdot \frac{t}{\Phi(1/t)},
\eeqlb
where $\lambda(x,y)$ is a non-negative function on $\mathbb{R}_{+}^{2}$ whose two-dimensional Laplace transform is given by
\beqlb\label{eqn3.11}
\iint_{\mathbb{R}_+^2} e^{-q_1 x-q_2 y}\lambda(x,y)\,dx\,dy
=\frac{q_1-q_2}{q_1^2 q_2^2(q_1^{\rho}-q_2^{\rho})}.
\eeqlb
From Theorem 2.3 in \cite{de1984domains}, 
the limit function $\lambda$ has a monotone density $d$ and such that as $t\to\infty$,
\beqlb\label{eqn3.5}
1-\mathbf{P}(\tau_{0}^{+}\leq tx,T_0-\tau_{0}^{+}\leq ty) \sim b^{-1} d(x,y)\cdot\frac{1}{t\Phi(1/t)}.
\eeqlb
Therefore, from Theorem \ref{Mainresult02}(1) and \eqref{eqn3.5}, we have as $t\to\infty$,
\beqnn
\frac{1-\mathbf{P}(\tau_{0}^{+}\leq at,T_0-\tau_{0}^{+}\leq t)}{\mathbf{P}(\tau_{0}^{+}>t)}\to\Gamma(\rho) d(1,a).
\eeqnn
Putting these results together, we see that  as $t\to \infty$,
$	\mathbf{P}(T_0-\tau_{0}^{+}>at \,|\, \tau_{0}^{+}>t)\to a^{\rho-1}+1-\Gamma(\rho) d(1,a).$
It remains to  prove that $d(1,a)$  equals to $f(a)$, where $f$ is defined in \eqref{eqn1.2}.
From \eqref{a3} and the fact that $\Phi(1/t)/t \in \mathrm{RV}_{\rho}^{\infty}$,  
it holds that 
$\lambda$ is  homogeneous of order $\rho+1$, that is, for any $a>0$, 
$
\lambda(ax,ay)=a^{\rho+1}\lambda(x,y).
$
Consequently, its density $d$  is  homogeneous of order $\rho-1$, namely
$
d(ax,ay) = a^{\rho-1} d(x,y)$,  which yields that  $d(x,y)=x^{\rho-1} d(1,y / x) .$  Setting  $q_1=q>0,\  q_2 =1$ in \eqref{eqn3.11} and then  making a  change of  variables  $x=x$ and $s=y/x$ on the left-side hand, we obtain  for any $q>0$,
\beqnn
\Gamma(\rho+1) \int_0^{\infty} \frac{d(1,s)}{(q+s)^{\rho+1}} ds=\frac{q-1}{q(q^\rho-1)}.
\eeqnn
Hence,  for any $q>0$,
\beqlb\label{eqn3.144}
\Gamma(\rho+1) \int_0^{\infty} \frac{d(1,s)-s^{\rho-1} / \Gamma(\rho)}{(q+s)^{\rho+1}} d s=\frac{q-1}{q(q^\rho-1)}-\frac{1}{q}=\frac{1-q^{\rho-1}}{q^\rho-1}=:D(q).
\eeqlb
By  Theorem 1 in  \cite[p.72]{berg2006quelques}, we  can verify that  $D(q)$  is a Stieltjes function.
From the inverse formula of Stielties transform, we have  for any $q>0$, 
\beqnn
D(q)=\frac{1-q^{\rho-1}}{q^\rho-1}=\int_0^{\infty} \frac{m(u)}{q+u} d u,
\eeqnn
where 
\beqnn
m(u):=\lim _{\varepsilon \rightarrow 0^{+}} \frac{-\mathrm{Im} D(-u+i\varepsilon)}{\pi} =\frac{\sin (\pi \rho)}{\pi} \frac{(1+u) u^{\rho-1}}{u^{2 \rho}-2 u^\rho \cos (\pi \rho)+1}.
\eeqnn
Therefore,    for any $q>0$, 
\beqlb\label{eqn3.15}
D(q)=\Gamma(\rho+1)\int_0^{\infty} \frac{h(s)/\Gamma(\rho)}{(q+s)^{\rho+1}} d s,
\eeqlb
where 
\beqnn
h(s):=\frac{\sin (\pi \rho)}{\pi} \int_0^s \frac{(s-u)^{\rho-1} (1+u) u^{\rho-1}}{u^{2 \rho}-2 u^\rho \cos (\pi \rho)+1} d u.
\eeqnn
By \eqref{eqn3.144} and \eqref{eqn3.15}, it follows from the uniqueness of the generalized Stieltjes transform that
\beqlb\label{eqn3.14}
d(1,s)=s^{\rho-1}/\Gamma(\rho)
+h(s)/\Gamma(\rho)
\eeqlb
for almost every $s>0$.
Recall that  $d(1,s)$ is decreasing in $s$ and the function in the  right-hand side of  \eqref{eqn3.14} is continuous, so \eqref{eqn3.14} holds for any $s>0$.  
Finally, we show that   $ a^{\rho-1}+1-\Gamma(\rho) f(a)=1-h(a)\in (0,1)$.
Using
\beqnn
\frac{(1+u) u^{\rho-1}}{u^{2 \rho}-2 u^\rho c+1} \leq u^{-\rho}, \quad u>0,
\eeqnn
we obtain
$0<h(a)< \frac{\sin (\pi \rho)}{\pi} \int_0^a(a-u)^{\rho-1} u^{-\rho} d u=1,
$
which completes the proof. \qed

Before proving Theorem \ref{MainThm01}(2), we  give the following  useful equations: for any bounded function $F$ and $G$
\beqlb
\int_{0}^{\infty}G(x)dx\int_{0}^{\infty}F(y)\nu(dy+x)\ar = \ar \int_{0}^{\infty}G(x)dx\int_{x}^{\infty}F(z-x)\nu(dz) \label{eqn3.1}\\  
\ar=\ar \int_{0}^{\infty}F(w)dw\int_{w}^{\infty}G(z-w)\nu(dz), \label{eqn3.2}
\eeqlb
where the first equation is obtained by making a change of variables $z=y+x$, $x=x$, and the second equation follows by $z=z$, $w=z-x$.

\noindent\textbf{Proof of Theorem \ref{MainThm01}(2).}  
By Lemma \ref{lem3.2}, we obtain
\beqlb\label{eqn3.101}
\mathbf{P}\big(H_0>ah \,|\, \widehat{H}_0>h\big)
=\frac{   \int_{0}^{\infty}\mathbf{P}_x(H_0>ah)dx\int_{0}^{\infty}\mathbf{P}_{y}(H_0>h)\nu(dy+x)}{\int_{0}^{\infty}\mathbf{P}_{y}(H_0>h)\overline{\nu}(y)dy}.
\eeqlb
On the one hand, from the proof of Theorem \ref{Mainresult02}(2) that  as $h\to\infty$,
\beqlb\label{a8}
\int_{0}^{\infty}\mathbf{P}_{y}(H_0>h)\overline{\nu}(y)dy\sim \Gamma(2-\alpha)\Gamma(\alpha-1)\cdot h\overline{\nu}(h).
\eeqlb
On the other hand, for any $\delta>0$, the numerator of \eqref{eqn3.101} can be decomposed into seven terms: $\sum_{i=1}^{7}K_i^{\delta}$, where   
$
K_i^{\delta}:=\iint_{A_i}\mathbf{P}_{x}(H_0>ah)\mathbf{P}_{y}(H_0>h)\nu(dy+x)dx
$
and  $\{A_i \, ; 1\leq i\leq 7\}$  is   a partition of $\mathbb{R}_{+}^{2}$; see Figure~\ref{Figure02}.

\begin{center}
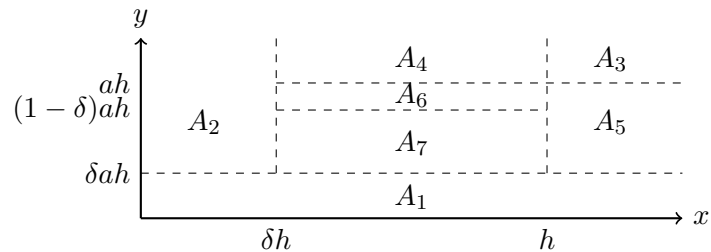

	\begin{tikzpicture}[scale=0.4, xscale=3,yscale=1]
		\def\d{1.5}      
		\def\H{4.5}      
		\def\Hm{3.6}     
		\draw[->, thick] (0,0) -- (6,0) node[right] {$x$};
		\draw[->, thick] (0,0) -- (0,6) node[above] {$y$};
		
		\draw[dashed] (\d,\d) -- (\d,6);
		\draw[dashed] (\H,\d) -- (\H,6);
		
		\draw[dashed] (0,\d) -- (6,\d);
		\draw[dashed] (\d,\Hm) -- (\H,\Hm);
		\draw[dashed] (\d,\H) -- (6,\H);
		
		\node[left]  at (0,\d) {$\delta ah$};
		\node[left]  at (0,\Hm) {$(1-\delta)ah$};
		\node[left]  at (0,\H) {$ah$};
		
		\node[below] at (\d,0) {$\delta h$};
		\node[below] at (\H,0) {$h$};
		
		\node at (3,0.7) {$A_1$};
		
		\node at (0.7,3.2) {$A_2$};
		\node at (3,2.5) {$A_7$};
		\node at (5.2,3.2) {$A_5$};
		
		\node at (3,4.1) {$A_6$};
		
		\node at (3,5.3) {$A_4$};
		\node at (5.2,5.3) {$A_3$};
	\end{tikzpicture}
	\captionof{figure}{\small Partition of $\mathbb{R}_{+}^2$ into seven regions. Here 
		$A_1=\{0<x<\infty, 0<y<\delta ah\}$, 
		$A_2=\{0<x<\delta h, y>\delta ah\}$, 
		$A_3=\{x>h, y>ah\}$, 
		$A_4=\{\delta h<x<h, y>ah\}$, 
		$A_5=\{x>h, \delta ah<y<ah\}$, 
		$A_6=\{\delta h<x<h, (1-\delta)ah<y< ah\}$, 
		$A_7=\{\delta h<x<h, \delta ah<y<(1-\delta) ah\}$.\label{Figure02}}
\end{center}
We first prove that $K_1^{\delta}$, $K_2^{\delta}$ and $K_6^{\delta}$ are $o(h\overline{\nu}(h))$, as $h\to\infty$ and then $\delta\to 0$.
Applying \eqref{eqn3.2} with $F(\cdot)=\mathbf{P}_{\cdot}(H_0>ah)\cdot\mathbf{1}_{[0,\delta ah]}(\cdot)$ and $G\equiv 1$, we have
\beqnn
K_1^{\delta}\leq \int_{0}^{\infty}dx\int_{0}^{\delta ah}\mathbf{P}_{y}(H_0>ah)\nu(dy+x)=\int_{0}^{\delta ah}\mathbf{P}_{w}(H_0>ah)\overline{\nu}(w)dw.
\eeqnn
We can also notice that
\beqnn
K_2^{\delta}\leq\int_{0}^{\delta h}\mathbf{P}_{x}(H_0>h)dx\int_{0}^{\infty}\nu(dy+x)=\int_{0}^{\delta h}\mathbf{P}_{x}(H_0>h)\overline{\nu}(x)dx.
\eeqnn
Thus from \eqref{eqn4.28}, we obtain $K_1^{\delta}+K_2^{\delta}=o(h\overline{\nu}(h))$ as $h\to\infty $ and then $\delta \to 0+$.
For $K_6^{\delta}$, we can apply \eqref{eqn3.2} with $F(\cdot)=\mathbf{P}_{\cdot}(H_0>ah)\cdot\mathbf{1}_{[(1-\delta )ah,ah]}(\cdot)$ and $G\equiv 1$,
\beqnn
K_6^{\delta} \leq \int_{0}^{\infty}dx\int_{(1-\delta )ah}^{ah}\mathbf{P}_{y}(H_0>ah)\nu(dy+x)=\int_{(1-\delta)ah}^{ah}\mathbf{P}_{w}(H_0>ah)\overline{\nu}(w)dw\leq\delta ah\, \overline{\nu}((1-\delta)ah),
\eeqnn
which implies that $K_6^{\delta}=o(h\overline{\nu}(h))$ as $h\to\infty $ and then $\delta \to 0+$.
Next, we estimate the remaining terms respectively. It follows from   Proposition~1.5.9 in \cite[p.~27]{BinghamGoldieTeugels1987} that as $h\to\infty$, 
\beqnn
K_3^{\delta}=\int_{h}^{\infty}dx\int_{ah}^{\infty}\nu(dy+x)=\int_{(a+1)h}^{\infty}\overline{\nu}(u)du\to\frac{(a+1)^{1-\alpha}}{\alpha-1}.
\eeqnn
We now turn to consider $K_4^{\delta}$ and $K_5^{\delta}$, which can be written as
\beqnn
K_4^{\delta}\ar=\ar\int_{\delta h}^{h}\mathbf{P}_{x}(H_0>h)\overline{\nu}(ah+x)dx\quad\text{and}\quad K_5^{\delta}=\int_{\delta a h}^{ah}\mathbf{P}_{x}(H_0>ah)\overline{\nu}(h+x)dx.
\eeqnn
The expression for $K_4^{\delta}$ is straightforward, while   $K_5^{\delta}$ is handled by applying  \eqref{eqn3.2} with $F(\cdot)=\mathbf{P}_{\cdot}(H_0>ah)\cdot\mathbf{1}_{[\delta ah, ah]}(\cdot)$, $G=\mathbf{1}_{[h,\infty)}(\cdot)$.
It follows from \eqref{eqn2.5} and \eqref{a5} that 
\beqnn
\lim_{\delta\to 0}\lim_{h\to\infty}\frac{K_4^{\delta}}{h\overline{\nu}(h)}
\ar=\ar \lim_{\delta\to 0}\lim_{h\to\infty}\int_{\delta }^{1}\big(1-W(h(1-x))/W(h)\big)\cdot\overline{\nu}(h(x+a))/\overline{\nu}(h)dx\cr\cr
\ar=\ar  \int_{0}^{1}(1-(1-x)^{\alpha-1})(x+a)^{-\alpha}dx=
-\frac{(a+1)^{1-\alpha}}{\alpha-1}+\int_{0}^{1}(x+a)^{1-\alpha}(1-x)^{\alpha-2}dx.
\eeqnn
and 
\beqnn
\lim_{\delta\to 0}\lim_{h\to\infty}\frac{K_5^{\delta}}{h\overline{\nu}(h)}
\ar=\ar \lim_{\delta\to 0}\lim_{h\to\infty}a\int_{\delta }^{1}\big(1-W(ah(1-x))/W(ah)\big)\cdot\overline{\nu}(h(1+ax))/\overline{\nu}(h)dx\cr\cr \ar=\ar a\int_{0}^{1}(1-(1-x)^{\alpha-1})(1+ax)^{-\alpha}dx=
-\frac{(a+1)^{1-\alpha}}{\alpha-1}+\int_{0}^{1}(ax+1)^{1-\alpha}(1-x)^{\alpha-2}dx.
\eeqnn
Finally, we estimate $K_7^{\delta}$. Applying \eqref{eqn3.1} with $F(\cdot)=\mathbf{P}_{\cdot}(H_0>ah)\cdot\mathbf{1}_{[\delta ah,(1-\delta) ah]}(\cdot)$ and $G(\cdot)=\mathbf{P}_{\cdot}(H_0>h)\cdot\mathbf{1}_{[\delta h,(1-\delta) h]}(\cdot)$, we have
\beqnn
K_7^{\delta}\ar=\ar\int_{\delta h}^{h}\mathbf{P}_x(H_0>h)dx\int_{x+\delta ah}^{x+(1-\delta)ah}\mathbf{P}_{z-x}(H_0>ah)\nu(dz)\cr\cr
\ar=\ar h\int_{\delta}^{1}\mathbf{P}_{xh}(H_0>h)dx\int_{x+a\delta}^{x+a(1-\delta)}\mathbf{P}_{(z-x)h}(H_0>ah)\nu(h\cdot dz).
\eeqnn
Denote the inner integral by
\beqnn
I:=\int_{x+a\delta}^{x+a(1-\delta)}\mathbf{P}_{(z-x)h}(H_0>ah)\nu(h\cdot dz)=\int_{x+a\delta}^{x+a(1-\delta)}\big(1-W(ah-(z-x)h)/W(ah)\big)\nu(h\cdot dz),
\eeqnn
where the second equality follows from \eqref{eqn2.5}.
Note that by  \eqref{a5}, we have as $h\to\infty$, uniformly in $z\in[u+a\delta,u+a(1-\delta)]$, 
\beqnn
1-W(ah-(z-x)h)/W(ah)\to 1-\big(1-(z-x)/a\big)^{\alpha-1}
\quad \text{and}\quad 
\overline{\nu}(hz)/\overline{\nu}(h)\to  z^{-\alpha}.
\eeqnn
This allows us to apply the Helly-Bray theorem to get that as $h\to\infty$,
\beqnn
\frac{I}{\overline{\nu}(h)}=\int_{x+a\delta}^{x+a(1-\delta)}\mathbf{P}_{(z-x)h}(H_0>ah)\nu(h\cdot dz)/\overline{\nu}(h)\to\alpha\int_{x+a\delta}^{x+a(1-\delta)}\big[1-\big(1-(z-x)/a\big)^{\alpha-1}\big]z^{-\alpha-1}dz.
\eeqnn
Moreover,   $I$ can be bounded by
$\overline{\nu}\big(h(x+a\delta)\big)-\overline{\nu}\big(h(x+a(1-\delta))\big)$,
thus we can apply the dominated convergence theorem to $K_7^{\delta}$ to obtain that 
\beqnn
\lim_{\delta\to 0}\lim_{h\to\infty}\frac{K_7^{\delta}}{h\overline{\nu}(h)}\ar=\ar\alpha\int_{0}^{1}(1-(1-x)^{\alpha-1})dx\int_{x}^{x+a}\big[1-\big(1-(z-x)/a\big)^{\alpha-1}\big]z^{-\alpha-1}dz\cr\cr\ar=\ar
a\alpha\int_{0}^{1}(1-(1-x)^{\alpha-1})dx\int_{0}^{1}(1-(1-y)^{\alpha-1})(ay+x)^{-\alpha-1}dy.
\eeqnn
Using the identity
$
1-(1-x)^{\alpha-1}=(\alpha-1)\int_{0}^{x}(1-s)^{\alpha-2}ds,
$
the integral in the last line can be rewritten as
\beqnn
\lefteqn{a\alpha(\alpha-1)^2\int_{0}^{1}dx\int_{0}^{1}(ay+x)^{-\alpha-1}dy\int_{0}^{x}(1-s)^{\alpha-2}ds\int_{0}^{y}(1-t)^{\alpha-2}dt}\cr\cr\ar=\ar
a\alpha(\alpha-1)^2\int_{0}^{1}(1-s)^{\alpha-2}ds\int_{0}^{1}(1-t)^{\alpha-2}dt\int_{s}^{1}dx\int_{t}^{1}(ay+x)^{-\alpha-1}dy\cr\cr\ar=\ar
(\alpha-1)\int_{0}^{1}(1-s)^{\alpha-2}ds\int_{0}^{1}(1-t)^{\alpha-2} \big[(a+1)^{1-\alpha}+(at+s)^{1-\alpha}-(at+1)^{1-\alpha}-(a+s)^{1-\alpha}\big]dt.
\eeqnn
We finally get that
\beqnn
\lim_{\delta\to 0}\lim_{h\to\infty}\frac{K_7^{\delta}}{h\overline{\nu}(h)}\ar=\ar \frac{(a+1)^{1-\alpha}}{\alpha-1}+(\alpha-1)\int_{0}^{1}x^{\alpha-2}dx\int_{0}^{1}y^{\alpha-2}(a+1-ay-x)^{1-\alpha}dy\cr\cr\ar\ar
-\int_{0}^{1}(ax+1)^{1-\alpha}(1-x)^{\alpha-2}dx-\int_{0}^{1}(x+a)^{1-\alpha}(1-x)^{\alpha-2}dx.
\eeqnn
Putting these estimates together, we get that as $h\to\infty$,
\beqlb\label{a6}
\int_{0}^{\infty}\mathbf{P}_x(H_0>ah)dx\int_{0}^{\infty}\mathbf{P}_{y}(H_0>h)\sim
(\alpha-1)\int_{0}^{1}x^{\alpha-2}dx\int_{0}^{1}y^{\alpha-2}(a+1-ay-x)^{1-\alpha}dy\cdot  h\overline{\nu}(h).
\eeqlb
Substituting \eqref{a8} and \eqref{a6} into \eqref{eqn3.101}, we get that 
$
\lim_{h\to\infty}\mathbf{P}\big(H_0>ah \,|\, \widehat{H}_0>h\big)=g(a),
$
where  $g$ is defined in \eqref{eqn1.3}.
Moreover, it is easy to verify that $0<g(a)<\frac{\alpha-1}{\Gamma(2 -\alpha) \Gamma(\alpha-1)} \int_0^1 x^{\alpha-2} d x \int_0^1 y^{\alpha-2}(1-y)^{1-\alpha} d y=1.$
The desired result holds.
\qed
%

\subsection{The  transient case.}
This section is devoted to the proofs of Theorems~\ref{Mainresult005} and \ref{MainThm02}. We first establish two auxiliary lemmas on asymptotic behavior and uniform upper bounds for the tail probabilities of  $\tau_0^-$ and $H_0$.

\begin{lemma}\label{Lemma303}
	Under Condition \ref{Assumption02}, we have  as $t\to\infty$, 
	$
	\mathbf{P}_x(  \tau_0^->t )\sim  \mathbf{P}_x(  H_0>\beta t)
	\sim x \overline\nu(\beta t)/\beta.
	$
\end{lemma}
\proof
For the first claim, by Theorem~2.2 in   \cite{DenisovShneer2013}, as $t\to\infty$, 
$
\mathbf{P}_{x}\big( \tau_0^->t\big) \sim \mathbf{E}_{x}[\tau_0^-] \cdot \overline{\nu}(\beta t). 
$
Moreover, from Proposition~17 in \cite[p.172]{Bertoin1996}, we  have  
$
\mathbf{E}_{x}[\tau_0^-]=x/\beta,$
which proves the first claim.  For the second claim,  it suffices to prove that  
$
\mathbf{P}_x(H_0>\beta t ) \sim \mathbf{P}_x(\tau_0^->t),
$
as $t\to \infty$.
On the one hand, notice that
\beqnn
\mathbf{P}_x(\tau_0^->t ,H_0\leq \beta t) \leq \mathbf{P}_x(\tau_0^->t ,\tau_{\beta t}^+>t )\leq \mathbf{P}_x(\tau_0^->t ,\mathcal{J}^{\beta t}>t),
\eeqnn where $
\mathcal{J}^{\beta t}:=\inf \{s \geq 0: \Delta X_s>\beta t\}.$
From this and Theorem~3.4 in \cite{Xu2021a}, we have 
$ \mathbf{P}_x(\tau_0^->t ,H_0\leq \beta t )=o( \mathbf{P}_x(\tau_0^->t  )),$ as $t\to\infty$
and hence
\beqlb\label{less}
\mathbf{P}_x( H_0> \beta t )\geq  \mathbf{P}_x(\tau_0^->t ,H_0> \beta t )\sim \mathbf{P}_x(\tau_0^->t  ).
\eeqlb
On the other hand, for any $\varepsilon\in(0,1)$ we have 
$
\mathbf{P}_x(\tau_0^-<(1-\varepsilon)t, H_0\geq \beta t )  =  \mathbf{P}_x( \tau_0^-<(1-\varepsilon)t, \tau_{\beta t}^+<\tau_0^- ).
$
Conditioned on the event $\{\tau_{\beta t}^+<\tau_0^-\}$, we can split $ \tau_0^-$ into $\tau_{\beta t}^+ $ and $\tau_0^- -\tau_{\beta t}^+ $, where $\tau_0^- -\tau_{\beta t}^+ $ is the downward crossing time of $X$ at $0$ after upcrossing at $\beta t$. 
From this and the strong Markov property of $X$,
\beqnn
\mathbf{P}_x(\tau_0^-<(1-\varepsilon)t, H_0\geq \beta t )   \ar\leq\ar \mathbf{P}_x\big(\tau_0^- -\tau_{\beta t}^+ <(1-\varepsilon)t, \inf_{0\le s\le \tau_{\beta t}^+} X_s>0 \big) \cr
\ar\ar\cr
\ar\leq\ar  \mathbf{P}_x(\tilde\tau_0^-(\beta t) <(1-\varepsilon)t, \inf_{0\le s\le \tau_{\beta t}^+} X_s>0 )= \mathbf{P}_{\beta t}( \tau_0^- <(1-\varepsilon)t)\cdot \mathbf{P}_x( H_0\geq \beta t  ), 
\eeqnn
where $\tilde\tau_0^-(\beta t) $ is the independent copy of the downward crossing time  of $X$ at $0$ starting from $\beta t$. 
Hence, 
\beqnn
\frac{\mathbf{P}_x( H_0\geq \beta t  )}{\mathbf{P}_x(\tau_0^->(1-\varepsilon)t)}\leq \frac{\mathbf{P}_{\beta t}( \tau_0^- <(1-\varepsilon)t)\cdot \mathbf{P}_x( H_0\geq \beta t  )}{\mathbf{P}_x(\tau_0^->(1-\varepsilon)t)}+1.
\eeqnn
Note that for any $z>\beta$, as $t\to\infty$, $X_t+(z+\beta)t/2\to\infty$ a.s., we have
\beqlb\label{3.84}
\mathbf{P}_{zt}( \tau_0^- > t)=\mathbf{P}\big(\inf_{0\leq s\leq t}X_s+(z+\beta)t/2>-(z-\beta) t/2\big)\to 1.
\eeqlb
From this we get that as $t\to \infty$, 
$
1-\mathbf{P}_{\beta t}( \tau_0^- <(1-\varepsilon)t)=\mathbf{P}_{\beta t}( \tau_0^- \geq(1-\varepsilon)t)\to 1.
$ Consequently, as $t\to\infty$,
\beqlb\label{eqn3.24}
\frac{\mathbf{P}_x( H_0\geq \beta t  )}{\mathbf{P}_x(\tau_0^->(1-\varepsilon)t)}\leq \frac{1}{1-\mathbf{P}_{\beta t}( \tau_0^- <(1-\varepsilon)t)}\to 1.
\eeqlb
It follows from \eqref{eqn3.24} and the asymptotic behavior of $\mathbf{P}_x(  \tau_0^->t )$ that 
\beqnn
\limsup_{t\to\infty}\frac{\mathbf{P}_x( H_0\geq \beta t  )}{\mathbf{P}_x(\tau_0^->t)}=\limsup_{t\to\infty}\frac{\mathbf{P}_x( H_0\geq \beta t  )}{\mathbf{P}_x(\tau_0^->(1-\varepsilon)t)}\cdot \frac{\mathbf{P}_x(\tau_0^->(1-\varepsilon)t)}{\mathbf{P}_x(\tau_0^->t)} \leq (1-\epsilon)^{-\alpha},
\eeqnn
which goes to $1$, by letting $\epsilon\to0$.
Combining this with (\ref{less}), the lemma is proved.
\qed

\begin{lemma}\label{uniest}
	Under Condition \ref{Assumption02}, for any $0<\delta<\beta $,  there exists a constant $C_\delta>0$ such that for any $t>0  $ and $x \in[0, \delta t]$, we have
	$
	\mathbf{P}_x(\tau_0^- >t)+\mathbf{P}_x(H_0>\beta t)\leq C_\delta(1+x)  \overline{\nu}(\beta t).
	$
\end{lemma}

\proof 
Let $\{\tau_{-1,i}^{-}\}_{i \geq 1}$ be a sequence of i.i.d. copies of $\tau_{-1}^{-}$. 
By the strong Markov property of $X$, it holds that 
$
\mathbf{P}_x\big(\tau_{0}^{-} > t\big) 
=\mathbf{P}\big(\tau_{-x}^{-} > t\big)   
\leq \mathbf{P}\big(\sum_{i=1}^{[x]+1} \tau_{-1,i}^{-}>t\big)  . 
$
For any constant $d>1$ such that $d \cdot \delta \cdot \mathbf{E}[\tau_{-1}^{-}]<1$, we have for all $t>0$ and $x \in[0, \delta t]$, 
\beqnn
\mathbf{P}\Big(\sum_{i=1}^{[x]+1} \tau_{-1,i}^{-} >t\Big) 
\ar=\ar \mathbf{P}\Big(\sum_{i=1}^{[x]+1}\big(\tau_{-1,i}^{-}-\gamma \cdot  \mathbf{E}[\tau_{-1}^{-}]\big) > t- ([x]+1) \cdot \gamma \cdot  \mathbf{E}[\tau_{-1}^{-}]  \Big) \cr
\ar\leq\ar  \mathbf{P}\Big(\sum_{i=1}^{[x]+1}\big(\tau_{-1,i}^{-}- d \cdot  \mathbf{E}[\tau_{-1}^{-}]\big) > t- (\delta t+1) \cdot \gamma \cdot  \mathbf{E}[\tau_{-1}^{-}]  \Big).
\eeqnn 
Since as $t\to \infty $, $t- (\delta t+1) \cdot d \cdot  \mathbf{E}[\tau_{-1}^{-}] \sim \big(1-  d\cdot \delta\cdot  \mathbf{E}[\tau_{-1}^{-}] \big)\cdot t \to \infty $, 
by Theorem 2 in \cite{DenisovFossKorshunov2010} there exists a constant $C>0$ such that for large $t$, the last probability is bounded uniformly in $x \geq 0$  by 
$
C (x+1) \cdot \mathbf{P}\big(\tau_{-1}^{-} \geq \big(1- d \cdot \delta \cdot  \mathbf{E}[\tau_{-1}^{-}] \big)\cdot t \big)  ,
$
which, by Lemma \ref{Lemma303} (1) and Condition~\ref{Assumption02}, is asymptotically equivalent as $t\to\infty$ to 
\beqnn
\frac{C}{\beta} (x+1) \cdot \overline{\nu}\big( \big(1- \delta \cdot \theta\cdot  \mathbf{E}[\tau_{-1}^{-}] \big)\cdot t \big)
\sim  \frac{C}{\beta} (x+1) \cdot  \Big(\frac{1-\delta d /\beta }{\beta}\Big)^{-\theta} \cdot \overline{\nu}(\beta t).
\eeqnn
Consequently, we get the  uniform upper bound for $\tau^{-}_{0}$.
Combining this with the uniform estimate \eqref{eqn3.24}, we deduce that  the same uniform upper bound for $H_0$ also holds.
\qed

\subsubsection{Proof of Theorem \ref{Mainresult005}} 
From  \eqref{aa2.9} and  Lemma~\ref{lem3.2}, we have for any $\delta\in (0,\beta)$,
\beqnn
\mathbf{P}(  \tau_0^{+}>t| \tau_0^{+}<\infty)=\frac{1}{b-\beta}\int_{0}^{\infty}\mathbf{P}_{x}(  \tau_{0}^->t )\overline{\nu}(x)dx=\frac{1}{b-\beta}\sum_{i=1}^{3}J_i^{\delta},
\eeqnn
where 
$J_1^{\delta}=\int_{0}^{(\beta-\delta)t}\mathbf{P}_{x}(  \tau_0^->t )\overline{\nu}(x)dx$, $J_2^{\delta}=\int_{(\beta-\delta)t}^{(\beta+\delta)t}\mathbf{P}_{x}(  \tau_{0}^->t )\overline{\nu}(x)dx$ and $J_3^{\delta}=\int_{(\beta+\delta)t}^{\infty}\mathbf{P}_{x}(  \tau_{0}^->t )\overline{\nu}(x)dx.$
From  Lemma~\ref{uniest},  we can use the dominated convergence theorem to $J_1^{\delta}$  and then apply the  Lemma~\ref{Lemma303}  to obtain that   
$
J_1^{\delta}= o\big(\beta t\overline{\nu}(\beta t)\big),
$
as $t\to\infty$.
For $J_2^{\delta}$, it can be bounded by $\int_{(\beta-\delta)t}^{(\beta+\delta)t}\overline{\nu}(x)dx$.  Hence， from  \eqref{eqn2.15},  we obtain 
\beqnn
\lim_{\delta\to 0+}\limsup_{t\to\infty}\frac{J_2^{\delta}}{\beta t\overline{\nu}(\beta t)} \leq \lim_{\delta\to 0+}\frac{1}{\theta-1} \big[\big(1-\delta/\beta\big)^{1-\theta}-\big(1+\delta/\beta\big)^{1-\theta}\big]=0.
\eeqnn
It remains to consider  $J_3^{\delta}$.
Note that
$
\mathbf{P}_{(\beta+\delta)t}(\tau_0^->t)\cdot\int_{(\beta+\delta)t}^{\infty}\overline{\nu}(x)dx\leq J_3^{\delta}\leq\int_{(\beta+\delta)t}^{\infty}\overline{\nu}(x)dx.
$
From \eqref{3.84}, we know that as $t\to \infty$, $\mathbf{P}_{(\beta+\delta)t}(\tau_0^->t)\to 1$. Therefore, it follows from  \eqref{eqn2.15} that as $h\to\infty $ and then $\delta \to 0+$, $J_3^{\delta}\sim \beta t\overline{\nu}(\beta t)/(\theta-1)$.
Putting these estimates together, we see the  first claim holds. We now turn to prove the second claim. 
From Lemma \ref{lem3.2} and \eqref{aa2.9}, we have 
\beqnn
\mathbf{P}\big(H_0>h| \tau_0^{+}<\infty\big)
= \frac{1}{b-\beta} \int_{0}^{\infty}\mathbf{P}_{x}(H_0>h)\overline{\nu}(x)dx,
\eeqnn
where the above integral  can be decomposed into:
$\int_0^h \mathbf{P}_{x}(  H_0>h  ) \overline{\nu}(x)dx$ and  $\int_h^{\infty} \mathbf{P}_{x}(  H_0>h ) \overline{\nu}(x)dx.
$
For the first term, it follows from \eqref{eqn3.24} that for any $\epsilon>0$,
\beqnn
\limsup_{h\to\infty}\frac{1}{h\overline{\nu}(h)}\int_0^h \mathbf{P}_{x}(  H_0>h )\overline{\nu}(x)dx \ar\leq\ar \limsup_{t\to\infty}\frac{1}{\beta t\overline{\nu}(\beta t)}\int_0^{\beta t}\mathbf{P}_{x}(  \tau_0^->(1-\varepsilon )t )\overline{\nu}(x)dx\\
\ar=\ar \limsup_{t\to\infty}\int_0^{(1-\varepsilon)^{-1}\beta t}\frac{\mathbf{P}_{x}(  \tau_0^->t )}{(1-\varepsilon)^{-1}\beta t\, \overline{\nu}((1-\varepsilon)^{-1}\beta t)}\overline{\nu}(x)dx.
\eeqnn
As in the proof of the first claim, the latter limit vanishes, as $\varepsilon\to 0$.  Hence   
$
\int_0^h \mathbf{P}_{x}(  H_0>h  ) \overline{\nu}(x)dx=o(h\overline{\nu}(h)),
$
as $h\to\infty$ and then $\varepsilon\to 0$.
For the second term, using \eqref{eqn2.15} again, we immediately obtain that  as $h\to\infty$,
$
\int_h^{\infty} \mathbf{P}_{x}(  H_0>h ) \overline{\nu}(x)dx=\int_h^{\infty}  \overline{\nu}(x)dx\sim h\overline{\nu}(h)/(\theta-1).
$
Combining these two results together, we have 
proved the desired result.\qed
\subsubsection{Proof of Theorem \ref{MainThm02}}
\noindent\textbf{Proof of Theorem \ref{MainThm02}(1).}
By Lemma \ref{lem3.2}, we obtain
\beqlb\label{eqn3.4}
\mathbf{P}\big(T_0-\tau_{0}^{+}>at \,|\, \tau_{0}^{+}>t, \tau_0^{+}<\infty\big)
=\frac{ \int_{0}^{\infty}\mathbf{P}_{x}(\tau_{0}^{-}>at)dx\int_{0}^{\infty}\mathbf{P}_{y}(\tau_{0}^{-}>t)\nu(dy+x)}{\int_{0}^{\infty}\mathbf{P}_{y}(\tau_{0}^{-}>t)\overline{\nu}(y)dy}.
\eeqlb
On the one hand, from the proof of Theorem \ref{Mainresult005} (1),  as $t\to\infty$,
\beqlb\label{a7}
\int_{0}^{\infty}\mathbf{P}_{y}(\tau_{0}^{-}>t)\overline{\nu}(y)dy\sim \frac{\beta t\overline{\nu}(\beta t)}{\theta-1}.
\eeqlb
On the other hand, the numerator of \eqref{eqn3.4} can be decomposed into the next two terms
\beqlb\label{eqn3.6}
\iint_{H_1}\mathbf{P}_{x}(\tau_{0}^{-}>at)\mathbf{P}_{y}(\tau_{0}^{-}>t)\nu(dy+x)dx \quad\text{and}\quad \iint_{H_2}\mathbf{P}_{x}(\tau_{0}^{-}>at)\mathbf{P}_{y}(\tau_{0}^{-}>t)\nu(dy+x)dx 
\eeqlb
where with $\delta\in (0,\beta)$,
$
H_1:=\big\{(x,y)\, |\,x>(\beta+\delta)at,  \   y>(\beta+\delta)t \big\}\ \text{and}\ 
H_2:=\mathbb{R}_{+}\times\mathbb{R}_{+}/H_{1}.
$
The  first term can be bounded below and above  by 
\beqnn  \mathbf{P}_{(\beta+\delta)at}(\tau_{0}^{-}>at)\cdot \mathbf{P}_{(\beta+\delta)t}(\tau_{0}^{-}>t)\cdot\int_{(\beta+\delta)(1+a)t}^{\infty}\overline{\nu}(x)dx\quad\text{and}\quad   \int_{(\beta+\delta)(1+a)t}^{\infty}\overline{\nu}(x)dx,
\eeqnn
respectively.
From \eqref{3.84}, we know that as $t\to \infty$, $\mathbf{P}_{(\beta+\delta)at}(\tau_{0}^{-}>at)\cdot \mathbf{P}_{(\beta+\delta)t}(\tau_{0}^{-}>t)\to1 $. Hence, by  \eqref{eqn2.15} that 
\beqlb\label{eqn3.7}
\lefteqn{\lim_{\delta\to 0+}\lim_{t\to\infty}\frac{ 1}{\beta t\overline{\nu}(\beta t)} \iint_{H_1}\mathbf{P}_{x}(\tau_{0}^{-}>at)\mathbf{P}_{y}(\tau_{0}^{-}>t)\nu(dy+x)dx}\ar\ar\cr\cr
\ar=\ar \lim_{\delta\to 0+}\lim_{t\to\infty}\frac{ 1}{\beta t\overline{\nu}(\beta t)}\int_{(\beta+\delta)(1+a)t}^{\infty}\overline{\nu}(x)dx=\frac{ (1+a)^{-\theta+1}}{\theta-1}.
\eeqlb
We now consider the second term in \eqref{eqn3.6}. Observe that
$H_2\subset  H_{21}\cup H_{22},$
where 
\beqnn
\ar\ar H_{21}:=\big\{(x,y)\in\big(0,(\beta-\delta)a t\big)\times \mathbb{R}_{+}\  \text{or}\   (x,y)\in\mathbb{R}_{+}\times \big(0,(\beta-\delta)t\big)  \big\}\quad\text{and}\cr\cr
\ar\ar H_{22}:=\big\{(x,y)\in \big((\beta-\delta)a t,(\beta+\delta)a t\big)\times \mathbb{R}_{+} \  \text{or}\   (x,y)\in\mathbb{R}_{+}\times \big((\beta-\delta)t,(\beta+\delta)t\big)  \big\}.
\eeqnn
From Lemma \ref{uniest} (i), we obtain  
\beqnn
\lefteqn{\limsup_{t\to\infty}\frac{1}{\beta t\overline{\nu}(\beta t)}  \iint_{H_{21}}\mathbf{P}_{x}(\tau_{0}^{-}>at)\mathbf{P}_{y}(\tau_{0}^{-}>t)\nu(dy+x)dx}\ar\ar\cr\cr
\ar\leq\ar
C\lim_{t\to\infty}\Big(\int_{0}^{\infty} \mathbf{P}_{y}(\tau_{0}^{-}>t)\,  \overline{\nu}(y)dy+\int_{0}^{\infty} \mathbf{P}_{x}(\tau_{0}^{-}>at\big)\,  \overline{\nu}(x)dx\Big)=0.
\eeqnn
Furthermore, by Proposition 1.5.10 of \cite[p.27]{BinghamGoldieTeugels1987},
\beqnn
\lefteqn{\lim_{\delta\to0+}\limsup_{t\to\infty}\frac{1}{\beta t\overline{\nu}(\beta t)}\iint_{H_{22}}\mathbf{P}_x(\tau_{0}^{-}>at)\cdot \mathbf{P}_y(\tau_{0}^{-}>  t)\,\nu(dy+x)dx}\ar\ar\cr\cr
\ar\leq\ar \lim_{\delta\to0+}\lim_{t\to\infty}\frac{1}{\beta t\overline{\nu}(\beta t)} \Big(\int_{(\beta-\delta)a t}^{(\beta+\delta)a t} \overline{\nu}(x)dx+\int_{(\beta-\delta)t}^{(\beta+\delta)t}  \overline{\nu}(y)dy \Big)=0.
\eeqnn
These two estimates imply that the  second term in \eqref{eqn3.6} is $o(\beta t \cdot \overline{\nu}(\beta t))$ as $t\to\infty$ and then $\delta\to 0+$.
Combining this with \eqref{a7}, \eqref{eqn3.7} and substituting into \eqref{eqn3.4} yields the desired result.
\qed

\noindent\textbf{Proof of Theorem \ref{MainThm02}(2).}
Here we  mainly prove the first conditional probability  and  the equality  follows directly from Lemma \ref{lem3.2}. 
By Lemma \ref{lem3.2}, 
\beqlb\label{eqn3.10}
\mathbf{P}\big(\widehat{H}_0>ah \,|\, H_0>h,\tau_{0}^{+}<\infty\big)
=\frac{   \int_{0}^{\infty}\mathbf{P}_x(H_0>ah)dx\int_{0}^{\infty}\mathbf{P}_{y}(H_0>h)\nu(dy+x)}{\int_{0}^{\infty}\mathbf{P}_{y}(H_0>h)\overline{\nu}(y)dy}.
\eeqlb
From  the proof of Theorem \ref{Mainresult005} (2) that  as $h\to\infty$,  
\beqnn
\int_{0}^{\infty}\mathbf{P}_{y}(H_0>h)\overline{\nu}(y)dy\sim \frac{h\overline{\nu}(h)}{\theta-1}.
\eeqnn
Arguing as in the proof of the first claim, we similarly obtain that as $h\to\infty$,
\beqnn
\int_{0}^{\infty}\mathbf{P}_x(H_0>ah)dx\int_{0}^{\infty}\mathbf{P}_{y}(H_0>h)\nu(dy+x)\sim \frac{(1+a)^{1-\theta}}{\theta-1}\cdot h\overline{\nu}(h).
\eeqnn
Substituting these results into \eqref{eqn3.10} completes the proof.
\qed

\bigskip
\textbf{Acknowledgment.} The authors would like to thank Wei Xu for his enlightening comments.

 \bibliographystyle{plain}
 
 \bibliography{Reference}

 \end{document}